\documentclass[11pt]{article}
\usepackage{amsmath,amsthm,amssymb,stackengine}
\usepackage[hyperfootnotes=false]{hyperref}
\usepackage{color}
\usepackage{cancel}
\usepackage{titlesec}
\usepackage{enumitem}
\usepackage{cases}
\usepackage{caption}
\titleformat{\paragraph}
{\normalfont\normalsize\bfseries}{\theparagraph}{1em}{}
\titlespacing*{\paragraph}
{0pt}{3.25ex plus 1ex minus .2ex}{1.5ex plus .2ex}

\def\titlerunning#1{\gdef\titrun{#1}}
\makeatletter
\def\author#1{\gdef\autrun{\def\and{\unskip, }#1}\gdef\@author{#1}}
\def\address#1{{\def\and{\\\hspace*{18pt}}\renewcommand{\thefootnote}{}%
\footnote {#1}}%
\markboth{\autrun}{\titrun}}
\makeatother
\def\email#1{\hspace*{4pt}{\em e-mail}: #1}
\def\MSC#1{{\renewcommand{\thefootnote}{}%
\footnote{\emph{Mathematics Subject Classification (2020):} #1}}}
\def\keywords#1{\par\medskip
\noindent\textbf{Keywords:} #1}

\newtheorem{theorem}{Theorem}[section]

\newtheorem{prop}[theorem]{Proposition}
\newtheorem{cor}[theorem]{Corollary}
\newtheorem{lemma}[theorem]{Lemma}

\theoremstyle{definition}

\newtheorem{remark}[theorem]{Remark}

\numberwithin{equation}{section}

\def\s{\mathfrak s}

\def\cA{\mathcal A}

\def\cC{\mathcal C}

\def\cL{\mathcal L}
\def\cM{\mathcal M}

\def\cO{\mathcal O}
\def\cP{\mathcal P}
\def\cQ{\mathcal Q}

\def\cR{\mathcal R}
\def\cS{\mathcal S}

\def\cW{\mathcal W}
\def\cX{\mathcal X}
\def\cY{\mathcal Y}
\def\cZ{\mathcal Z}
\def\PG{{\rm PG}}

\def\F{{\mathbb F}}

\def\PGL{{\rm PGL}}
\def\PGO{{\rm PGO}}

\def\rk{{\rm rk}}
\def\Tr{{\rm Tr}}

\def\v{\boldsymbol v}
\def\u{\boldsymbol u}

\def\w{\omega}
\def\z{\zeta}

\begin{document}


\baselineskip=16pt

\titlerunning{}

\title{On the geometry of a $(q+1)$-arc of $\PG(3, q)$, $q$ even}

\author{Michela Ceria
\and
Francesco Pavese}

\date{}

\maketitle

\address{M. Ceria, F. Pavese: Dipartimento di Meccanica, Matematica e Management, Politecnico di Bari, Via Orabona 4, 70125 Bari, Italy; \email{\{michela.ceria, francesco.pavese\}@poliba.it}
}


\MSC{Primary 05E18. Secondary 05B25; 51E20; 51N35; 14N10.}

\begin{abstract}
In $\PG(3, q)$, $q = 2^n$, $n \ge 3$, let $\cA = \{(1,t,t^{2^h},t^{2^h+1}) \mid t \in \F_q\} \cup \{(0,0,0,1)\}$, with $\gcd(n,h) = 1$, be a $(q+1)$-arc and let $G_h \simeq \PGL(2, q)$ be the stabilizer of $\cA$ in $\PGL(4, q)$. The $G_h$-orbits on points, lines and planes of $\PG(3, q)$, together with the point-plane incidence matrix with respect to the $G_h$-orbits on points and planes of $\PG(3, q)$ are determined. The point-line incidence matrix with respect to the $G_1$-orbits on points and lines of $\PG(3, q)$ is also considered. In particular, for a line $\ell$ belonging to a given line $G_1$-orbits, say $\cL$, the point $G_1$-orbit distribution of $\ell$ is either explicitly computed or it is shown to depend on the number of elements $x$ in $\F_q$ (or in a subset of $\F_q$) such that $\Tr_{q|2}(g(x)) = 0$, where $g$ is an $\F_q$-map determined by $\cL$.   

\end{abstract}

\section{Introduction}

Let $\PG(3, q)$ be the three-dimensional projective space over the finite field $\F_q$ with $q$ elements. A {\em $(q+1)$-arc} of $\PG(3, q)$ is a set of $q+1$ points of $\PG(3, q)$ no four in a plane. If $q = 2^n$, a $(q+1)$-arc of $\PG(3, q)$ is projectively equivalent to the pointset 
\begin{align}
    & \cA = \{P_t = (1,t,t^{2^h},t^{2^h+1}) \mid t \in \F_q\} \cup \{(0,0,0,1)\}, \label{arc}
\end{align}
where $\gcd(n, h) = 1$. If $h = 1$, then $\cA$ is called {\em twisted cubic} and if $q$ is odd every $(q+1)$-arc of $\PG(3, q)$ is a {\em twisted cubic}. Moreover $\cA$ and $\{(1,t,t^{2^{n-h}},t^{2^{n
-h}+1}) \mid t \in \F_q\} \cup \{(0,0,0,1)\}$ are projectively equivalent. The subgroup of $\PGL(4, q)$ stabilizing $\cA$, say $G_h$, is isomorphic to $\PGL(2, q)$ whenever $q > 4$, and contains a subgroup isomorphic to $\PGL(2, q)$ if $q \le 4$. For more details on $(q+1)$-arcs of $\PG(3, q)$ the reader is referred to \cite[Chapter 21]{H2}. 

Several authors dealt with the action of the group $G_1$ on points, lines and planes of $\PG(3, q)$ and the determination of their related incidence matrices. Indeed, a comprehensive understanding of this topic is of some interest in purely geometric terms \cite{DMP1, DMP2, DMP3, DMP}; in particular the classification of $G_1$-orbits on lines of $\PG(3,q)$ is equivalent to the classification of pencils of cubics in $\PG(1, q)$, see \cite{GL}. Moreover aspects of these problems are related and have been applied in coding theory in order to obtain asymptotically optimal multiple covering codes \cite{b2020}, or to study the weight distribution of the cosets of $\cC$ \cite{DMP0}, or to compute the coset leader weight enumerator of $\cC$ \cite{BPS}, where $\cC$ is the generalized Reed-Solomon code associated with a twisted cubic.   

In this paper we assume $q = 2^n$. In Section~\ref{sec:points}, we show that, similarly to the case $h = 1$, the group $G_h$ has five orbits on points on $\PG(3, q)$ and that the point-plane incidence matrix with respect to the $G_h$-orbits on points and planes of $\PG(3, q)$ mirrors the case $h = 1$, which was computed in \cite{b2020}. If $q \not\equiv 0 \pmod{3}$, the group $G_h$ fixes a symplectic polarity of $\PG(3, q)$ and hence to each $G_h$-orbit on points there corresponds a $G_h$-orbit on planes of $\PG(3, q)$. 

In Section~\ref{sec:lines}, we show that the group $G_h$ has $2q+7+\xi$ orbits on lines of $\PG(3, q)$, providing a proof of \cite[Conjecture 8.2]{DMP} in the even characteristic case. Finally the point-line incidence matrix with respect to the $G_1$-orbits on points and lines of $\PG(3, q)$ is considered. For a line $\ell$ belonging to a given line $G_1$-orbits, say $\cL$, the point $G_1$-orbit distribution of $\ell$ is either explicitly computed or it is shown to depend on the number of elements $x$ in $\F_q$ (or in a subset of $\F_q$) such that $\Tr_{q|2}(g(x)) = 0$, where $g$ is an $\F_q$-map determined by $\cL$.

\section{Preliminaries}

Let $q=2^n$ and let $\F_q$ be the finite field of order $q$. Let $\PG(3, q)$ be the three-dimensional projective space over $\F_{q}$ equipped with homogeneous projective coordinates $(X_1, X_2, X_3, X_4)$. Denote by $U_i$ the point having $1$ in the $i$-th position and $0$ elsewhere. 

Let $\cW(3, q)$ be the symplectic polar space consisting of the subspaces of $\PG(3, q)$ represented by the totally isotropic subspaces of $\F_{q}^4$ with respect to the non-degenerate alternating form $\beta$ given by 
\begin{align*}
x_1 y_4 + x_2 y_3 + x_3 y_2 + x_4 y_1. 
\end{align*}
The subspaces of $\cW(3, q)$ of maximum dimension are lines and are called {\em generators}. Denote by $\s$ the symplectic polarity of $\PG(3, q)$ defining $\cW(3, q)$.

Let $h$ be a positive integer such that $\gcd(h,n) = 1$ and consider the $(q+1)$-arc of $\PG(3,q)$ given by \eqref{arc}. This means that $\cA$ consists of $q+1$ points of $\PG(3,q)$ such that no four of them are coplanar. For $h = 1$ or $h = n-1$, $\cA$ is a twisted cubic of $\PG(3, q)$. 
A line of $\PG(3,q)$ joining two distinct points of $\cA$ is called a {\em real chord}. Let $\bar{\cA}$ be given by
\begin{equation}\label{quadratic}
\begin{aligned}
& \{\bar{P}_t = (1, t, t^{2^h}, t^{2^h+1})\mid t \in \F_{q^2}\} \cup \{U_4\} && \mbox{ if } h \mbox{ is odd}, \\
& \{\bar{P}_t = (1, t, t^{2^{n+h}}, t^{2^{n+h}+1})\mid t \in \F_{q^2}\} \cup \{U_4\} && \mbox{ if } h \mbox{ is even}.
\end{aligned}
\end{equation}
It holds that $\bar{\cA}$ is a $(q^2+1)$-arc of $\PG(3, q^2)$ and $\cA$ is embedded in $\bar{\cA}$. The line of $\PG(3, q^2)$ obtained by joining $\bar{P}_t$ and $\bar{P}_{t^q}$, with $t \in \F_{q^2} \setminus \F_q$, meets the canonical Baer subgeometry $\PG(3, q)$ in the $q+1$ points of a line skew to $\cA$. Such a line is called an {\em imaginary chord}. If $r$ is a (real or imaginary) chord, then the line $r^{\s}$ is called ({\em real} or {\em imaginary}) {\em axis}. For each point $P\in \mathcal{A}$, the {\em tangent line} to $\mathcal{A}$ at $P$ is the line $\ell_P = \langle P, P' \rangle$, with $P'=(0,1,0,t^{2^h})$ if $P = P_t$ and $P' = U_3$ if $P = U_4$. At each point $P_t$ (resp. $U_4$) of $\cA$ there corresponds the {\em osculating plane} $P_t^{\s}$ (resp. $U_4^{\s}$) with equation $t^{2^h+1} X_1 + t^{2^h} X_2 + t X_3 + X_4 = 0$ (resp. $X_1 = 0$), meeting $\cA$ only at $P_t$ (resp. $U_4$) and containing the tangent line $\ell_{P_t}$ (resp. $\ell_{U_4}$). Hence the $q+1$ lines tangent to $\cA$ are generators of $\cW(3, q)$ and they form a regulus of the hyperbolic quadric $\cQ^+(3, q): X_1 X_4 + X_2 X_3 = 0$. There are $q(q+1)/2$ real chords and $q(q-1)/2$ imaginary chords. Tangents, real chords and imaginary chords are all the chords of $\cA$. 

The group $G_h$ of projectivities of $\PG(3, q)$ stabilizing $\cA$ is isomorphic to $\PGL(2, q)$ whenever $q \geq 5$, and the elements of $G_h$ are represented by the matrices
\begin{align*}
& M_{a,b,c,d}=
\begin{pmatrix}
a^{2^{h}+1} & a^{2^h}b  & ab^{2^h} & b^{2^{h}+1} \\
a^{2^h}c  & a^{2^h}d  & b^{2^h}c  & b^{2^h}d\\
ac^{2^h}  & bc^{2^h}  & ad^{2^h}  & bd^{2^h}\\
c^{2^{h}+1} & c^{2^h}d  & cd^{2^h} & d^{2^{h}+1}
\end{pmatrix},
\end{align*}
where $a,b,c,d \in \F_q$ and $ad+bc = 1$. This group leaves $\mathcal{W}(3,q)$ invariant, being $M_{a,b,c,d}^tJM_{a,b,c,d}=J$, where \[J=\begin{pmatrix}0&0&0&1\\0&0&1&0\\ 0 &1 &0&0 \\ 1 &0&0&0\end{pmatrix}\] is the Gram matrix of $\beta$. We will denote by $\bar{G}_h$ or $\bar{G}_{n+h}$ the subgroup of $\PGL(4, q^2)$ fixing $\bar{\cA}$, depending on $h$ is odd or even. In both cases their elements are represented by the matrices $M_{a,b,c,d}$, with $a,b,c,d \in \F_{q^2}$ and $ad+bc = 1$. For more properties on $(q+1)$-arcs in even characteristic, we refer the readers to \cite[Section 21.3]{H2}.

\section{Point orbits}\label{sec:points}
In this section, we study the orbits on points and planes under the action of $G_h$. We start by proving that each point of $\PG(3,q)\setminus \cA$ lies on exactly one chord of $\cA$, similarly to the case $h \in \{1, n-1\}$, see \cite[Theorem 21.1.9]{H2}.

\begin{prop}\label{corde}
Every point of $\PG(3,q) \setminus \cA$ lies on exactly one chord of $\cA$.
\end{prop}
\begin{proof}
The points on tangent lines form the hyperbolic quadric $\cQ^+(3,q)$ and they are $(q+1)^2$ in number.
 
If two real chords meet, then their intersection point belongs to $\cA$. Indeed, assume by contradiction that $s, s'$ are two real chords which meet in a point $P \notin \cA$, then these two lines generate a plane containing four points of $\cA$, contradicting the fact that $\cA$ is a $(q+1)$-arc. Every real chord has $q-1$ points not in common with $\cA$ and a real chord meets $\cQ^+(3, q)$ in points of $\cA$. Since there are $q(q+1)/2$ real chords, it follows that there are $(q^3-q)/2$ points of $\PG(3, q) \setminus \cA$ lying on real chords.

An imaginary chord of $\cA$, when extended in $\PG(3, q^2)$, meets $\bar{\cA} \setminus \cA$ in two points. Hence, two imaginary chords or a real chord and an imaginary one cannot meet outside $\cA$, otherwise, when extended in $\PG(3, q^2)$, they would span a plane of $\PG(3, q^2)$ containing four points of $\bar{\cA}$, a contradiction. Let $\cQ^+(3, q^2)$ denote the hyperbolic quadric of $\PG(3, q^2)$ such that $\cQ^+(3, q^2) \cap \PG(3, q) = \cQ^+(3, q)$. It is easily seen that an imaginary chord meets $\cQ^+(3, q^2)$ in points of $\bar{\cA} \setminus \cA$. Hence an imaginary chord has empty intersection with $\cQ^+(3,q)$. Since there are $q(q-1)/2$ imaginary chords, there are $(q^3-q)/2$ points of $\PG(3, q) \setminus \cA$ lying on imaginary chords. 
\end{proof}

\begin{prop}\label{points}
The group $G_h$ has five orbits on points of $\PG(3, q)$:
\begin{itemize}
    \item $\cA$ of size $q+1$;
    \item $\cO_0$ of size $(q^3-q)/3$, consisting of points not contained in any osculating plane;
    \item $\cO_1$ of size $(q^3-q)/2$, consisting of points on exactly one osculating plane;
    \item $\cO_2 = \cQ^+(3, q) \setminus \cA$ of size $q^2+q$, consisting of points on exactly two osculating planes;
    \item $\cO_3$ of size $(q^3-q)/6$, consisting of points on exactly three osculating planes.
\end{itemize}
\end{prop}
\begin{proof}
Taking into account Proposition~\ref{corde}, it is enough to determine the orbits of $G$ on points of $\PG(3, q) \setminus \cA$ lying on chords of $\cA$. Moreover the group $G_h$ acts transitively on tangent lines, real chords and imaginary chords of $\cA$. Let $\ell$ be a chord of $\cA$. 

If $\ell$ is tangent we choose $\ell = \langle U_1, U_2 \rangle$ and consider the point $U_2 \in \ell \setminus \cA$. The projectivities of $G_h$ stabilizing $U_2$ are represented by the matrices $M_{d^{-1},0,0,d}$, $d \in \F_q \setminus \{0\}$. Hence $|Stab_{G_h}(U_2)| = q-1$ and $|{U_2}^{G_h}| = q^2+q = |\cQ^+(3, q) \setminus \cA|$. The point $U_2$ belongs to $U_1^\s$ and $U_4^\s$ and to no other osculating plane. 

If $\ell$ is a real chord, we consider $\ell = \langle U_1, U_4 \rangle$. The stabilizer of $\ell$ in $G_h$ has order $2(q-1)$ and it is generated by the projectivities represented by the matrices $M_{0,1,1,0}$ and $M_{d^{-1},0,0,d}$, $d \in \F_q \setminus \{0\}$. Let $x \in \F_q \setminus \{0\}$ and let $Q = U_1 + x U_4 \in \ell \setminus \cA$. We will prove that 
\begin{align*}
    & |Stab_{G_h}(Q)| =  
    \begin{cases}
        2 & \mbox{ if } q \equiv -1 \pmod{3}, \\
        3 & \mbox{ if } q \equiv 1 \pmod{3} \mbox{ and } x \mbox{ is not a cube in } \F_q, \\
        6 & \mbox{ if } q \equiv 1 \pmod{3} \mbox{ and } x \mbox{ is a cube in } \F_q.
    \end{cases}
\end{align*}
A projectivity represented by $M_{d^{-1}, 0, 0, d}$ fixes $Q$ if and only if $d^{2^h+1} = 1$ which has $\gcd(2^h+1, 2^n-1)$ solutions. Since
\begin{align*}
    & \gcd(2^h+1, 2^n-1) = 
    \begin{cases}
    3 & \mbox{ if } n \mbox{ is even,} \\
    1 & \mbox{ if } n \mbox{ is odd,} 
    \end{cases}
\end{align*}
the number of these projectivities fixing $Q$ equals three or one according as $q \equiv 1$ or $-1 \pmod{3}$, respectively. Similarly, a projectivity represented by $M_{0, d^{-1}, d, 0}$ fixes $Q$ if and only if 
\begin{align}
    & d^{2^h+1} = x. \label{eq1}
\end{align}
Equation \eqref{eq1} admits a solution if and only if $x^{\frac{2^n-1}{\gcd(2^h+1, 2^n-1)}} = 1$ and in that case it has $\gcd(2^h+1, 2^n-1)$ solutions. Therefore, if $q \equiv 1 \pmod{3}$, the number of these projectivities fixing $Q$ is three or zero according as $x$ is a cube in $\F_q$ or not, whereas if $q \equiv -1 \pmod{3}$ exactly one of these projectivities leaves $Q$ invariant. The point $Q$ belongs to the osculating plane $P_t^\s$ if and only if $t^{2^h+1} = x$. Hence it lies on one osculating plane if $q \equiv -1 \pmod{3}$, whereas if $q \equiv 1 \pmod{3}$ on three or none osculating planes according as $x$ is a cube or not in $\F_q$. 

If $\ell$ is an imaginary chord, we consider the $(q^2+1)$-arc $\bar{\cA}$ of $\PG(3, q^2)$ given by \eqref{quadratic} and the non-canonical Baer subgeometry $\Sigma$ 
\begin{align*}
    & \{(\alpha, \beta, \beta^q, \alpha^q): \, \alpha, \beta \in \F_{q^2}, (\alpha,\beta) \neq (0,0)\}.
\end{align*}
Some straightforward computations show that the intersection of $\bar{\cA}$ with $\Sigma$ is the $(q+1)$-arc of $\Sigma$ consisting of the points $\bar{P}_t$, $t \in \F_{q^2}$, $t^{q+1} = 1$. The projectivities fixing $\bar{\cA} \cap \Sigma$ belong to $\bar{G}_h$, if $h$ is odd, or to $\bar{G}_{n+h}$, if $h$ is even, and are those represented by the matrices $M_{a,b,c,d}$ with $a,b,c,d \in \F_{q^2}$, $ad+bc=1$, $cd^q+ab^q=0$ and $a^{q+1}+b^{q+1}+c^{q+1}+d^{q+1}=0$. Denote by $S$ the group $Stab_{\bar{G}_h }( \Sigma) $ if $h$ is odd or $Stab_{\bar{G}_{n+h}}( \Sigma)$ if $h$ is even. Note that $G_h$ and $S$ are conjugated in $\bar{G}_h$ or $\bar{G}_{n+h}$. In this setting, $\ell = \langle U_1, U_4 \rangle_{q^2} \cap \Sigma$ is an imaginary chord of $\bar{\cA} \cap \Sigma$. The stabilizer of $\ell$ in $S$ has order $2(q+1)$ and it is generated by the projectivities represented by the matrices $M_{0,1,1,0}$ and $M_{d^{-1},0,0,d}$, $d \in \F_{q^2}$, with 
\begin{align}
    & d^{2^n+1} = 1. \label{eq2}
\end{align}
Fix $\alpha \in \F_{q^2} \setminus \{0\}$ and set $R = U_1 + \alpha^{q-1} U_4 \in \ell$. We show that 
\begin{align*}
    & |Stab_{S}(R)| =  
    \begin{cases}
        2 & \mbox{ if } q \equiv 1 \pmod{3}, \\
        3 & \mbox{ if } q \equiv -1 \pmod{3} \mbox{ and } \alpha^{q-1} \mbox{ is not a cube in } \F_{q^2}, \\
        6 & \mbox{ if } q \equiv -1 \pmod{3} \mbox{ and } \alpha^{q-1} \mbox{ is a cube in } \F_{q^2}.
    \end{cases}
\end{align*}
A projectivity represented by $M_{d^{-1}, 0, 0, d}$ fixes $R$ if and only if 
\begin{equation}\label{eq3}
    \begin{aligned}
        & d^{2^h+1} = 1 && \mbox{ if } h \mbox{ is odd,} \\
        & d^{2^{n+h}+1} = 1 && \mbox{ if } h \mbox{ is even.} 
    \end{aligned}
\end{equation}
Let $\bar{\w}$ be a primitive element of $\F_{q^2}$ and set $d = \bar{\w}^{i}$. Since $\gcd(2^h+1, 2^{2n}-1) = 3$ if $h$ is odd and $\gcd(2^{n+h}+1, 2^{2n}-1) = 3$ if $h$ is even, equation \eqref{eq3} is satisfied by
\begin{align*}
    & i \equiv \frac{t(2^{2n}-1)}{3} \pmod{2^{2n-1}}, t = 0,1,2.
\end{align*}
If $q \equiv -1 \pmod{3}$, these three solutions satisfy \eqref{eq2}, whereas if $q \equiv 1 \pmod{3}$, the unique solution fulfilling \eqref{eq2} occurs for $t=0$. Therefore there are three or one of these projectivities stabilizing $R$ depending on $q \equiv -1$ or $1 \pmod{3}$, respectively. Similarly, a projectivity represented by $M_{0, d^{-1}, d, 0}$ fixes $R$ if and only if 
\begin{equation}\label{eq4}
    \begin{aligned}
    & d^{-(2^h+1)} = \alpha^{q-1} && \mbox{ if } h \mbox{ is odd,} \\
    & d^{-(2^{n+h}+1)} = \alpha^{q-1} && \mbox{ if } h \mbox{ is even.} 
    \end{aligned}
\end{equation}
In this case, if $\alpha = \bar{\w}^k$, equation \eqref{eq4} admits a solution if and only if $3$ divides $k(2^n-1)$. Assume that $3$ divides $k(2^n-1)$, then \eqref{eq4} has the three solutions 
\begin{align*}
    & i \equiv \frac{-k(2^n-1) u_0 + t(2^{2n}-1)}{3} \pmod{2^{2n-1}}, t = 0,1,2,
\end{align*}
where $u_0$ is an integer such that $(2^h+1) u_0 + (2^{2n}-1) v_0 = 3$, if $h$ is odd and $(2^{n+h}+1) u_0 + (2^{2n}-1) v_0 = 3$, if $h$ is even. If $q \equiv -1 \pmod{3}$, these three solutions satisfy \eqref{eq2}, whereas if $q \equiv 1 \pmod{3}$, exactly one of these three solutions satisfies \eqref{eq2}. It follows that, if $q \equiv -1 \pmod{3}$, the number of these projectivities fixing $R$ is three or zero according as $\alpha^{q-1}$ is a cube in $\F_{q^2}$ or not, whereas if $q \equiv 1 \pmod{3}$ exactly one of these projectivities leaves invariant $R$. The point $R$ belongs to the osculating plane $\bar{P}_t^\s$ if and only if 
\begin{align*}
    & t^{2^h+1} = \alpha^{q-1} && \mbox{ if } h \mbox{ is odd,} \\
    & t^{2^{n+h}+1} = \alpha^{q-1} && \mbox{ if } h \mbox{ is even} 
\end{align*}
for some $t \in \F_{q^2}$, with $t^{q+1} = 1$. Hence it lies on one osculating plane if $q \equiv 1 \pmod{3}$, whereas if $q \equiv -1 \pmod{3}$ on three or none osculating planes according as $\alpha^{q-1}$ is a cube or not in $\F_{q^2}$.  
\end{proof}

\begin{cor}\label{chords}
If $q \equiv 1 \pmod{3}$, the points of $\cA \cup \cO_0 \cup \cO_3$ are on the real chords and those of $\cO_1$ on the imaginary chords. If $q \equiv -1 \pmod{3}$, the points of $\cO_0 \cup \cO_3$ are on the imaginary chords and those of $\cA \cup \cO_1$ on the real chords. 
\end{cor}

The $G_h$-orbits on planes of $\PG(3, q)$ can be described by using the symplectic polarity $\s$ of $\PG(3, q)$ defining $\cW(3, q)$. Let $\cO_j^\s = \{P^\s \mid P \in \cO_j\}$.

\begin{prop}
The group $G_h$ has five orbits on planes of $\PG(3, q)$:
\begin{itemize}
    \item The $q+1$ osculating planes;
    \item $\cO_0^\s$ of size $(q^3-q)/3$, consisting of planes containing no point of $\cA$;
    \item $\cO_1^\s$ of size $(q^3-q)/2$, consisting of planes containing one point of $\cA$;
    \item $\cO_2^\s$ of size $q^2+q$, consisting of planes containing two points of $\cA$;
    \item $\cO_3^\s$ of size $(q^3-q)/6$, consisting of planes containing three point of $\cA$.
\end{itemize}
\end{prop}

\subsection{Point plane distribution}
In the case when $\cA$ is a twisted cubic, the point plane distribution has been determined in \cite{b2020}. Most of the combinatorial arguments used in \cite{b2020} can be adapted to the general case of a $(q+1)$-arc $\cA$. Set $\cO_{q+1} = \cA$. Denote by $r_{i,j}$ the number of planes in $\cO_j^\s$ through a point in $\cO_i$ and by $k_{i,j}$ the number of points of $\cO_i$ in a plane of $\cO_j^\s$. Hence
\begin{align} \label{eq-main}
    & |\cO_i| r_{i,j} = |\cO_j^\s| k_{i,j}.
\end{align}
Note that 
\begin{align*}
    & r_{0,q+1} = 0, \; r_{q+1,q+1} = r_{1, q+1} = 1, \; r_{2, q+1} = 2, \; r_{3, q+1} = 3,
\end{align*}
by Proposition~\ref{points}. Since a plane meets $\cA$ in at most three points, taking into account Proposition~\ref{corde}, the following results can be easily deduced.
\begin{lemma}\label{lemma-chord}
Through a real chord of $\cA$ there pass $q-1$ planes of $\cO_3^\s$ and two planes of $\cO_2^\s$, whereas through an imaginary chord of $\cA$ there pass $q+1$ planes of $\cO_1^\s$.
\end{lemma}

\begin{lemma}\label{lemma-tec}
The following hold.
\begin{enumerate}
\item[1)] $r_{i, q+1} + r_{i,1} + 2r_{i,2} + 3r_{i,3} = (q+1)^2$, $i \in \{0,1,2,3\}$.
\item[2)] For $q \equiv 1 \pmod 3$, it holds
\begin{align*}
    & r_{2,2} + 3r_{2,3} = r_{1,2} + 3r_{1,3} = \frac{q(q+1)}{2}, \\
    & r_{0,2} + 3r_{0,3} = r_{3,2} + 3r_{3,3} = \frac{q(q+3)}{2}.
\end{align*}
\item[3)] For $q \equiv -1 \pmod 3$, it holds
\begin{align*}
    & r_{2,2} + 3r_{2,3} = r_{0,2}+3r_{0,3} = r_{3,2}+3r_{3,3} = \frac{q(q+1)}{2}, \\
    & r_{1,2} + 3r_{1,3} = \frac{q(q+3)}{2}.
\end{align*}
\end{enumerate}
\end{lemma}
\begin{proof}
To prove {\em 1)}, let $P$ be a point of $\PG(3, q) \setminus \cA$ and let us count in two ways the pairs $(Q, \sigma)$, where $Q$ is a point of $\cA$, $\sigma$ is a plane and the line $\langle P, Q \rangle$ is contained in $\sigma$. The point $Q$ can be chosen in $q+1$ ways and for a fixed $Q$ there are $q+1$ planes containing both $P$ and $Q$. Hence on the one hand there are $(q+1)^2$ pairs. On the other hand, if $P \in \cO_i$, where $i \in \{0,1,2,3\}$, then the number of these pairs equals $r_{i, q+1} + r_{i, 1} + 2r_{i,2} + 3r_{i,3}$. 

In order to show {\em 2)} and {\em 3)}, by Corollary~\ref{chords}, it is enough to verify that 
\begin{align*}
r_{i,2} + 3 r_{i,3} =
\begin{cases}
    \frac{q(q+1)}{2} & \mbox{ if } \cO_i \mbox{ consists of points not on real chords}, \\
    \frac{q(q+3)}{2} & \mbox{ if } \cO_i \mbox{ consists of points off } \cA \mbox{ on real chords}. 
\end{cases}
\end{align*}
Let $P$ be a point of $\PG(3, q) \setminus \cA$ and let us count in two ways the pairs $(\ell, \sigma)$, where $\ell$ is a real chord, $P \notin \ell$ and $\sigma = \langle P, \ell \rangle$ is a plane such that $|\sigma \cap \cA| \in \{2, 3\}$. If $P$ does not belong to any real chord, then each of the $q(q+1)/2$ real chords gives rise to a plane $\sigma$ which contains either one or three real chords. Hence equation $r_{i,2} + 3 r_{i,3} = q(q+1)/2$ follows. If $P$ belongs to the real chord $r$, then each of the $q(q+1)/2 - 1$ real chords distinct from $r$ spans with $P$ a plane $\sigma$. By Lemma~\ref{lemma-chord}, among these planes there are $q-1$ of them which contain precisely two chords distinct from $r$. Therefore in this case 
\begin{align*}
    & (r_{i,2}-2) + 3(r_{i,3} - (q-1)) + 2(q-1) = \frac{q(q+1)}{2} - 1
\end{align*}
holds true and equation $r_{i,2} + 3 r_{i,3} = q(q+3)/2$ follows. 
\end{proof}

\begin{prop}
\begin{align*}
    & r_{q+1, d} = 
    \begin{cases}
        1 & \mbox{ if } d = q+1, \\
        0 & \mbox{ if } d = 0, \\
        \frac{q(q-1)}{2} & \mbox{ if } d \in \{1,3\}, \\
        2q & \mbox{ if } d = 2.
    \end{cases}
\end{align*}
\end{prop}
\begin{proof}
The value $r_{q+1, 0}$, is easily determined. Through a point $Q \in \cA$ there are $q$ real chords. By Lemma~\ref{lemma-chord}, through a real chord there pass two planes of $\cO_2^\s$ and $q-1$ planes of $\cO_3^\s$. Therefore $r_{q+1, 2} = 2q$ and $r_{q+1, 3} = q(q-1)/2$, since a plane of $\cO_3$ contains two real chords through $Q$. Finally $r_{q+1, 1} = q^2+q+1 - (1 - 2q - q(q-1)/2) = q(q-1)/2$.
\end{proof}

\begin{prop}
If $q \equiv \xi \pmod 3$, $\xi \in \{\pm 1\}$, then 
\begin{align*}
    & r_{1,d} = 
    \begin{cases}
        1 & \mbox{ if } d=q+1, \\
        \frac{q(q-\xi)}{3} & \mbox{ if } d=0, \\
        \frac{q(q+\xi)}{2} & \mbox{ if } d=1, \\ 
        q & \mbox{ if } d=2, \\
        \frac{q(q-\xi)}{6} & \mbox{ if } d=3. \\
    \end{cases}
    & r_{2,d} = 
    \begin{cases}
        2 & \mbox{ if } d=q+1, \\
        \frac{q^2-1}{3} & \mbox{ if } d=0, \\
        \frac{q(q-1)}{2} & \mbox{ if } d=1, \\ 
        2q-1 & \mbox{ if } d=2, \\
        \frac{q^2-3q+2}{6} & \mbox{ if } d=3. \\
    \end{cases}
\end{align*}
\end{prop}
\begin{proof}
By Corollary~\ref{chords}, if $q \equiv 1 \pmod{3}$ (or $q \equiv -1 \pmod{3}$), $\cO_1$ consists of the points on imaginary chords (or, together with $\cA$, of the points on real chords). On the other hand, by Lemma~\ref{lemma-chord}, a plane $\sigma$ in $\cO_0^\s$ does not contain any imaginary (or real) chord and hence each of the $q(q-1)/2$ (or $q(q+1)/2$) imaginary (or real) chords of $\cA$ has exactly one point in common with $\sigma$. It follows that Equation~\ref{eq-main} reads as $\frac{q(q^2-1)}{2} r_{1,0} = \frac{q(q^2-1)}{3} \frac{q(q-\xi)}{2}$ and therefore $r_{1, 0} = q(q-\xi)/3$. As a consequence we have that $1 + r_{1,1}+r_{1,2}+r_{1,3} = q^2+q+1 - q(q-\xi)/3 = (2q^2+(3+\xi)q+3)/3$, which combined with Lemma~\ref{lemma-tec} {\em 1), 2)} (or Lemma~\ref{lemma-tec} {\em 1), 3)}), gives the required values of $r_{1,1}, r_{1,2}, r_{1,3}$.   

Analogously, no tangent line is contained in a plane $\sigma$ of $\cO_0^\s$ and hence each of the $q+1$ tangent lines has exactly one point in common with $\sigma$. In this case Equation~\ref{eq-main} becomes as $q(q+1) r_{2,0} = \frac{q(q^2-1)}{3} (q+1)$ and therefore $r_{2, 0} = (q^2-1)/3$. Moreover, $2 + r_{2,1} + r_{2,2} + r_{2,3} = q^2+q+1 - (q^2-1)/3 = (2q^2+3q+4)/3$, which together with Lemma~\ref{lemma-tec} {\em 1), 2), 3)} provides the values of $r_{2,1}, r_{2,2}, r_{2,3}$. 
\end{proof}

\begin{prop}
If $q \equiv \xi \pmod 3$, $\xi \in \{\pm 1\}$, then 
\begin{align*}
    & r_{0,d} = 
    \begin{cases}
        0 & \mbox{ if } d=q+1, \\
        \frac{q^2+\xi q+1}{3} & \mbox{ if } d=0, \\
        \frac{q(q-\xi)}{2} & \mbox{ if } d=1, \\ 
        q+1 & \mbox{ if } d=2, \\
        \frac{q^2+\xi q-2}{6} & \mbox{ if } d=3. \\
    \end{cases}
    & r_{3,d} = 
    \begin{cases}
        3 & \mbox{ if } d=q+1, \\
        \frac{q^2+\xi q-2}{3} & \mbox{ if } d=0, \\
        \frac{q(q-\xi)}{2} & \mbox{ if } d=1, \\ 
        q-2 & \mbox{ if } d=2, \\
        \frac{q^2+ \xi q+4}{6} & \mbox{ if } d=3. \\
    \end{cases}
\end{align*}
\end{prop}
\begin{proof}
First we determine $r_{3,3}$ and $r_{0,3}$. Fix a plane $\sigma$ in $\cO_3^\s$ and let us count in two ways the pairs $(Q, \gamma)$, where $Q$ is a point, $\gamma$ is an osculating plane and $Q \in \sigma \cap \gamma$. Since there are $q+1$ osculating planes and for each of these planes there are $q+1$ choices for the point $Q$, we have $(q+1)^2$ of such pairs. On the other hand, $\sigma$ contains $3$ points of $\cA$, $q(q-\xi)/2$ points of $\cO_1$, $q-2$ points of $\cO_2$ and $k_{3,3}$ points of $\cO_3$. It follows that 
\begin{align*}
    & 1 \times 3 + 1 \times \frac{q(q-\xi)}{2} + 2 \times (q-2) + 3 \times k_{3,3} = (q+1)^2.
\end{align*}
Therefore $k_{3,3} = (q^2+\xi q+4)/6$ and $k_{0, 3} = q^2+q+1 - 3 - q(q-\xi)/2 - (q-2) - (q^2+\xi q +4)/6 = (q^2+\xi q -2)/3$. By using Equation~\eqref{eq-main}, we get $r_{3,3} = (q^2+\xi q+4)/6$ and $r_{0,3} = (q^2+\xi q -2)/6$.

Next, the values $r_{0, 2}$, $r_{3, 2}$ can be recovered from Lemma~\ref{lemma-tec} {\em 2), 3)}, $r_{0, 1}$, $r_{3, 1}$ from Lemma~\ref{lemma-tec} {\em 1)} and $r_{0,0}$, $r_{3,0}$ by difference.
\end{proof}

The results so obtained are summarized in Tables \ref{rij} and \ref{kij}, where $q \equiv \xi \pmod 3$, $\xi \in \{\pm 1\}$.

\begin{table}[ht]\caption{The number of planes of $\cO_j^\s$ through a point of $\cO_i$}\label{rij}
\centering 
\begin{tabular}{|c|c|c|c|c|c|}
\hline 
$\quad$ & $\cA$ & $\cO_0$ & $\cO_1$  &  $\cO_2$ & $\cO_3$   \\
\hline 
$\cO_{q+1}^\s$ & $1$ & $0$ & $1$ & $2$ & $3$ \\
$\cO_0^\s$ & $0$ & $(q^2+\xi q+1)/3$ & $(q^2-\xi q)/3$ & $(q^2-1)/3$ & $(q^2+ \xi q-2)/3$ \\
$\cO_1^s$ & $(q^2-q)/2$ & $(q^2- \xi q)/2$ & $(q^2+ \xi q)/2$ & $(q^2-q)/2$ & $(q^2- \xi q)/2$ \\
$\cO_2^\s$ & $2q$ & $q+1$ & $q$ & $2q-1$ & $q-2$ \\
$\cO_3^\s$ & $(q^2-q)/2$ & $(q^2+\xi q -2)/6$ & $(q^2-\xi q)/6$ & $(q^2-3q+2)/6$ & $(q^2+\xi q+4)/6$ \\
\hline
\end{tabular} 
\end{table}

\begin{table}[ht]\caption{The number of points of $\cO_i$ in a plane of $\cO_j^\s$}\label{kij}
\centering 
\begin{tabular}{|c|c|c|c|c|c|}
\hline 
$\quad$ & $\cA$ & $\cO_0$  & $\cO_1$ & $\cO_2$ & $\cO_3$  \\
\hline 
$\cO_{q+1}^\s$ & $1$ & $0$ &$(q^2-q)/2$ & $2q$ & $(q^2-q)/2$  \\
$\cO_0^\s$ & $0$ & $(q^2+ \xi q+1)/3$ & $(q^2- \xi q)/2$ & $q+1$ & $(q^2+ \xi q-2)/6$  \\
$\cO_1^\s$ & $1$ & $(q^2- \xi q)/3$ & $(q^2+ \xi q)/2$ & $q$ & $(q^2- \xi q)/6$  \\
$\cO_2^\s$ & $2$ & $(q^2-1)/3$ & $(q^2-q)/2$ & $2q-1$ & $(q^2-3q+2)/6$  \\
$\cO_3^\s$ & $3$ & $(q^2+ \xi q-2)/3$ & $(q^2- \xi q)/2$ & $q-2$ & $(q^2+ \xi q+4)/6$ \\
\hline
\end{tabular} 
\end{table}

\section{Line orbits}\label{sec:lines}

In this section we determine the orbits of $G_h$ on the lines of $\PG(3, q)$. In order to do that we consider the Pl\"ucker correspondence, which sends the lines of $\PG(3,q)$ to the points of the Klein quadric 
\begin{align*}
    & \cQ^+(5,q): X_1 X_6 + X_2 X_5 + X_3 X_4 = 0.    
\end{align*}
Under this embedding, the $q+1$ lines tangent to $\cA$ are mapped to the $q+1$ points of the conic
\begin{align*}
    & \cX = \pi \cap \cQ^+(5, q) = \{(1,0,t, t, 0, t^2) \mid t \in \F_q\} \cup \{(0,0,0,0,0,1)\},
\end{align*}
whereas the lines of the opposite regulus of $\cQ^+(3, q)$ are mapped to the points of the conic
\begin{align*}
    & \cX' = \pi^\perp \cap \cQ^+(5, q) = \{(0,1,t, t,t^2,0) \mid t \in \F_q\} \cup \{(0,0,0,0,1,0)\}. 
\end{align*}
Here $\pi$ denotes the plane $X_2 = X_5 = X_3+X_4 = 0$ and $\perp$ is the polarity of $\PG(5, q)$ associated with $\cQ^+(5, q)$. The $(q+1)(q^2+1)$ generators of $\cW(3, q)$ are sent to the points of the parabolic quadric $\cQ(4, q) = \cQ^+(5, q) \cap \Pi$, where $\Pi: X_3+X_4 = 0$. Moreover, under the Pl\"ucker map, points of $\cW(3, q)$ correspond to lines of $\cQ(4, q)$ and, hence, if $\ell$ is a line of $\PG(3, q)$ that is not a generator of $\cW(3, q)$, then $\ell$ and $\ell^\perp$ are mapped to the two reguli of a hyperbolic quadric contained in $\cQ(4, q)$. To the group $G_h$, there corresponds a group $\tilde{G}_{h} \le \PGO^+(6, q)$ represented by the matrices  
\begin{align*}
    & \tilde{M}_{a,b,c,d} = 
    \begin{pmatrix} 
        a^{2^{h+1}} & 0 & a^{2^h}b^{2^h} & a^{2^h}b^{2^h} & 0 & b^{2^{h+1}} \\
        0 & a^2 & ab & ab & b^2 & 0 \\
        a^{2^h}c^{2^h} & ac & b^{2^h}c^{2^h} + bc + 1 & b^{2^h}c^{2^h} + bc & bd & b^{2^h}d^{2^h} \\ 
        a^{2^h}c^{2^h} & ac & b^{2^h}c^{2^h} + bc & b^{2^h}c^{2^h} + bc + 1 & bd & b^{2^h}d^{2^h} \\ 
        0 & c^2 & cd & cd & d^2 & 0 \\ 
        c^{2^{h+1}} & 0 & c^{2^h}d^{2^h} & c^{2^h}d^{2^h} & 0 & d^{2^{h+1}} 
    \end{pmatrix}
\end{align*}
with $a,b,c,d \in \F_q$, $ad+bc = 1$. Note that the group $\tilde{G}_h$ acts faithfully on both planes $\pi$ and $\pi^\perp$. Hence $\tilde{G}_h$ induces on $\pi$ (or $\pi^\perp$) a group isomorphic to $\PGL(2, q)$ acting in its natural representation on the $q+1$ point of $\cX$ (or $\cX'$). Furthermore, the group $\tilde{G}_h$ fixes the parabolic quadric $\cQ(4, q)$, each of the two conics $\cX$ and $\cX'$ and the hypersurface $\cY$ of $\Pi$:
\begin{align*}
& \cY = \{P \in \Pi \mid F_h(P) = 0\}, 
\end{align*}
where
\begin{align*}
& F_h(X_1,X_2,X_3,X_4,X_5,X_6) = X_1X_5^{2^h} + X_2^{2^h} X_6. 
\end{align*}
The hypersurface $\cY$ consists of $q+1$ planes through the point $\Pi^\perp = (0,0,1,1,0,0)$. Each of these planes intersects $\cQ(4, q)$ in exactly one of the following $q+1$ lines
\begin{align*}
& \langle (0,1, x, x, x^2, 0), (1,0,x^{2^h},x^{2^h},0,x^{2^{h+1}}) \rangle, \, x \in \F_q, \; \langle (0,0,0,0,1,0), (0,0,0,0,0,1) \rangle .
\end{align*}
These $q+1$ lines are obtained by joining a point of $\cX$ and a point of $\cX'$; they are pairwise disjoint and hence $|\cY \cap \cQ(4, q)| = (q+1)^2$. Taking into account Proposition~\ref{points}, the following result holds.

\begin{prop}
The group $\tilde{G}_h$ has five orbits on lines of $\cQ(4, q)$:
\begin{itemize}
    \item $\cL_{q+1}$ of size $q+1$, consisting of the lines of $\cY$;
    \item $\cL_0$ of size $(q^3-q)/3$, consisting of lines disjoint to $\cY$;
    \item $\cL_1$ of size $(q^3-q)/2$, consisting of lines meeting $\cY$ in one point;
    \item $\cL_2$ of size $q^2+q$, consisting of lines having two points in common with $\cY$;
    \item $\cL_3$ of size $(q^3-q)/6$, consisting of lines intersecting $\cY$ in three point.
\end{itemize}
\end{prop}

\begin{remark}
Note that a line of $\cL_1$ and $\cL_3$ has no point in common with $\cX \cup \cX'$, whereas each of the lines in $\cL_{q+1}$ and $\cL_2$ meets both $\cX$, $\cX'$ in one point. 
\end{remark}

The $G_h$-action on lines of $\PG(3, q)$ is equivalent to the $\tilde{G}_h$-action on points and reguli of $\cQ(4, q)$. Our goal is to describe these $\tilde{G}_h$-orbits. We first consider the $\tilde{G}_h$-orbits on points of $\cQ(4, q)$.

\begin{prop}\label{para}
Let $q \equiv \xi \pmod{3}$, $\xi \in \{\pm 1\}$. The group $\tilde{G}_h$ has $5+\xi$ orbits on points of $\cQ(4, q)$:
\begin{itemize}
\item $2$ orbits, $\cX$ and $\cX'$, of size $q+1$; 
\item one orbit, $\cY \cap \left(\cQ(4, q) \setminus (\cX \cup \cX')\right)$, of size $q^2-1$;
\item $2+\xi$ orbits on points of $\cQ(4, q) \setminus \cY$ of size $\frac{q^3-q}{2+\xi}$. 
\end{itemize}
\end{prop}
\begin{proof}
Let $P = U_5 + U_6$. Then $P \in \cY \cap \left(\cQ(4, q) \setminus (\cX \cup \cX')\right)$ and $Stab_{\tilde{G}_h}(P)$ is the group of order $q$ represented by the matrices $\tilde{M}_{1,0,c,1}$, with $c \in \F_q$. Hence $|P^{\tilde{G}_h}| = q^2-1 = |\cY \cap \left(\cQ(4, q) \setminus (\cX \cup \cX')\right)|$. 

Let $P = U_1 + U_5$. In this case $Stab_{\tilde{G}_h}(P)$ is the group of order $2+\xi$ represented by the matrices $\tilde{M}_{d^{-1}, 0, 0, d}$, with $d \in \F_q$ and $d^{2^h+1} = 1$. Hence the group $\tilde{G}_h$ has one or three orbits on points of $\cQ(4, q) \setminus \cY$ according as $\xi$ equals $-1$ or $1$, respectively. In the latter case the representatives of the orbits are $U_1 + U_5$, $U_1 + \w U_5$ and $U_1 + \w^2 U_5$, where $\w$ is a primitive element of $\F_q$.
\end{proof} 

In order to determine the $\tilde{G}_h$-orbits on reguli contained in $\cQ(4, q)$, the following results are needed.

\begin{lemma}\label{secarcs}
Let $s$ be a line of $\pi$ secant to $\cX$. The stabilizer of $s$ in $\tilde{G}_h$ has $(q+2)/2$ orbits on lines of $\pi$ secant to $\cX$:
\begin{itemize}
\item $\{s\}$;
\item $(q-2)/2$ orbits of size $q-1$; 
\item one orbit of size $2(q-1)$.
\end{itemize}
\end{lemma}
\begin{proof}
Since $\tilde{G}_h$ is transitive on the $q(q+1)/2$ lines of $\pi$ secant $\cX$ we may take $s$ as the line joining $U_1$ and $U_6$. Then $Stab_{\tilde{G}_h}(s)$ is the group of order $2(q-1)$ represented by $\left\langle \tilde{M}_{0,1,1,0}, \tilde{M}_{d^{-1}, 0,0, d} \right\rangle$, $d \in \F_q$. It is easily seen that $Stab_{\tilde{G}_h}(s)$ permutes in a single orbit the $2(q-1)$ lines of $\pi$ secant to $\cX$ and containing $U_1$ or $U_6$. Indeed each line containing $U_1$ (resp. $U_6$) is sent to a line containing $U_6$ (resp. $U_1$) by $\tilde{M}_{0,1,1,0}$, while it is sent to a different line  containing $U_1$ (resp. $U_6$) by $\tilde{M}_{d^{-1}, 0,0, d}$.
 
Every line $s'$ of $\pi$ secant to $\cX$ such that $s \cap s' \notin \cX$ belongs to an orbit of $Stab_{\tilde{G}_h}(s)$ of size $q-1$; indeed the unique non-trivial element of $Stab_{\tilde{G}_h}(s)$ fixing $s'$ is the involutory elation of $\pi$ having as axis the line $s$ and as center the point $s \cap s'$.
\end{proof}

In a similar fashion, one can prove the following results.

\begin{lemma}\label{extarcs}
Let $e$ be a line of $\pi$ external to $\cX$. The stabilizer of $e$ in $\tilde{G}_h$ has $q/2$ orbits on lines of $\pi$ external to $\cX$:
\begin{itemize}
\item $\{e\}$;
\item $(q-2)/2$ orbits of size $q+1$.
\end{itemize}
\end{lemma} 

\begin{theorem}\label{klein}
The group $\tilde{G}_h$ has $2(q+1)$ orbits on reguli of $\cQ(4, q)$:
\begin{itemize}
\item $2$ orbits of size $q(q-1)/2$;
\item $2$ orbits of size $q(q+1)/2$;
\item $2(q-2)$ orbits of size $(q^3-q)/2$;
\item $2$ orbits of size $q^3-q$.
\end{itemize}
\end{theorem} 
\begin{proof}
Let $\Sigma$ be a solid of $\Pi$ such that $\Sigma \cap \cQ(4, q)$ is a three-dimensional hyperbolic quadric $\cQ^+$. Then $\Sigma$ meets $\pi$ and $\pi^\perp$ in a line. Let $\Sigma \cap \pi$ and $\Sigma \cap \pi^\perp$ be denoted by $\ell$ and $\ell'$, respectively, and $Stab_{\tilde{G}_h}(\Sigma) = Stab_{\tilde{G}_h}(\ell \cup \ell') = Stab_{\tilde{G}_h}(\ell) \cap Stab_{\tilde{G}_h}(\ell')$. The three lines $\ell$, $\ell'$ and $\Sigma^\perp$ are pairwise orthogonal. It follows that the lines $\ell$, $\ell'$ are either both secant or both external to $\cQ(4, q)$. 

Assume that the former case occurs. Since $\tilde{G}_h$ is transitive on the $q(q+1)/2$ lines of $\pi^\perp$ secant $\cX'$, we may fix $\ell'$. Hence $Stab_{\tilde{G}_h}(\ell')$ is a group of order $2(q-1)$ fixing a line of $\pi$ secant to $\cX$. It follows that $\ell$ has to lie in one of the $(q+2)/2$ orbits of Lemma~\ref{secarcs}. Since every member of $Stab_{\tilde{G}_h}(\ell \cup \ell')$ stabilizes the reguli of $\cQ^+$, the number of $\tilde{G}_h$-orbits on reguli of $\cQ(4, q)$ obtained in this way is twice the number of orbits of Lemma~\ref{secarcs}. Moreover their size is $q(q+1)/2$ times the size of the orbits of Lemma~\ref{secarcs}.

Similarly, if the latter case occurs, we may fix $\ell'$ to be our favorite line of $\pi^\perp$ external to $\cX'$. In this case $Stab_{\tilde{G}_h}(\ell')$ is a group of order $2(q+1)$ fixing a line of $\pi$ external to $\cX$ and $\ell$ has to be in one of the $q/2$ orbits of Lemma~\ref{extarcs}. Again it can be checked that every member of $Stab_{\tilde{G}_h}(\ell \cup \ell')$ fixes the reguli of $\cQ^+$. Therefore the number of $\tilde{G}_h$-orbits on reguli of $\cQ(4, q)$ obtained in this way is twice the number of orbits of Lemma~\ref{extarcs} and their size is $q(q-1)/2$ times the size of the orbits of Lemma~\ref{extarcs}.
\end{proof}

With the same notation used in the proof of Theorem~\ref{klein}, denote by $\cS$ one of the $2(q+1)$ orbits on reguli of $\cQ(4, q)$ and by $\Sigma$ a solid of $\Pi$ such that 
\begin{itemize}
\item[{\em i)}] $\Sigma \cap \cQ(4, q)$ is a three-dimensional hyperbolic quadric $\cQ^+$;
\item[{\em ii)}] exactly one of the two reguli $\cR_1$, $\cR_2$ of $\cQ^+$ belongs to $\cS$. 
\end{itemize}
The next results, which can be deduced from the proof of Theorem~\ref{klein}, provides a solid $\Sigma$ such that one of the two reguli $\cR_1$, $\cR_2$ of $\cQ^+ = \Sigma \cap \cQ(4, q)$ is a representative for the orbit $\cS$.

\begin{cor}\label{representatives}
Let $\zeta$ be a fixed element in $\F_{q^2} \setminus \F_q$ such that $(\z+1)^{q+1} = 1$.
\begin{itemize}
\item $|\cS| = q(q+1)/2$, $\Sigma: X_3 = 0$, $\cQ^+: X_1 X_6 + X_2 X_5 = 0$.
\item $|\cS| = q^3-q$, $\Sigma: X_1 = X_3$, $\cQ^+: X_1^2 + X_1 X_6 + X_2 X_5 = 0$.
\item $|\cS| = (q^3-q)/2$ and $|\Sigma \cap \cX| = 2$, $\Sigma: X_1 + \frac{z^4}{z^2+1} X_3 + X_6 = 0$ for some $z \in \F_q \setminus \{0,1\}$, $\cQ^+: X_1^2 + \frac{z^4}{z^2+1} X_1 X_3 + X_3^2 + X_2X_5 = 0$.
\item $|\cS| = q(q-1)/2$, $\Sigma: X_1 + \z^{q+1} X_2 + \z^{2(q+1)} X_3 + \z^{q+1} X_5 + X_6 = 0$, \\
$\cQ^+: X_1^2 + \z^{q+1} X_1X_2 + \z^{2(q+1)} X_1X_3 + \z^{q+1} X_1 X_5 + X_2 X_5 + X_3^2 = 0$.
\item $|\cS| = (q^3-q)/2$ and $|\Sigma \cap \cX| = 0$, $\Sigma: X_1 + \frac{t^{q+1}}{\z^{q+1}} X_2 + t^{q+1} X_3 + \frac{t^{q+1}}{\z^{q+1}} X_5 + X_6 = 0$, for some $t \in \F_{q^2} \setminus \left\{0, \z^2, \z^{2q}\right\}$, $(t+1)^{q+1} = 1$, $\cQ^+: X_1^2 + \frac{t^{q+1}}{\z^{q+1}} X_1X_2 + t^{q+1} X_1X_3 + \frac{t^{q+1}}{\z^{q+1}} X_1 X_5 + X_2 X_5 + X_3^2 = 0$.
\end{itemize}
\end{cor}

The results of Proposition~\ref{para} and Theorem~\ref{klein} provide a proof of \cite[Conjecture 8.2]{DMP} in the even characteristic case.

\begin{theorem}
Let $q$ be even, with $q \equiv \xi \pmod{3}$, $\xi \in \{\pm 1\}$. The group $G_h$ has $2q+7+\xi$ orbits on lines of $\PG(3, q)$:
\begin{itemize}
\item $2$ orbits of size $q+1$; 
\item $2$ orbits of size $q(q-1)/2$;
\item $2$ orbits of size $q(q+1)/2$;
\item one orbit of size $q^2-1$;
\item $2+\xi$ orbits of size $\frac{q^3-q}{2+\xi}$;
\item $2(q-2)$ orbits of size $(q^3-q)/2$;
\item $2$ orbits of size $q^3-q$.
\end{itemize}
\end{theorem}

\subsection{Point line distribution}

In this section we focus on the case when $\cA$ is a twisted cubic (i.e. $h = 1$) and we address the point line distribution with respect to $G_1$-orbits on points and lines of $\PG(3, q)$. Dually, this problem is equivalent to determine 
\begin{itemize}
\item the number of lines of $\cL_i$ through a point of $\cQ(4, q)$ of a given $\tilde{G}_1$-orbit on points of $\cQ(4, q)$;
\item the number of lines of $\cL_i$ in a regulus of a given $\tilde{G}_1$-orbit on reguli of $\cQ(4, q)$.  
\end{itemize}
As for $\tilde{G}_1$-orbits on points of $\cQ(4, q)$, recall that besides the three three orbits $\cX$, $\cX'$ and $\cY \setminus (\cX \cup \cX')$, there is one or three more orbits, according as $q \equiv -1$ or $1 \pmod{3}$. Let us call $\cZ_i$, $i = 1, \dots, 2 + \xi$, these $\tilde{G}_1$-orbits. If $q \equiv 1 \pmod{3}$, from the proof of Proposition~\ref{para}, it follows that $\cZ_i = R_i^{\tilde{G}_1}$, where $R_{i} = U_1 + \w^{i-1} U_5$, $i = 1,2,3$. In particular $|\cZ_i| = (q^3-q)/3$.  

\subsubsection{The number of lines of $\cL_i$ through a point of $\cQ(4, q)$}

In order to determine the number of lines of $\cL_i$ through a point of $\cP$, where $\cP$ is a $\tilde{G}_1$-orbits of points of $\cQ(4, q)$, we calculate the number $k$ of points of $\cP$ on a line of $\cL_i$ and then use the formula 
\[
\frac{|\cL_i| k}{|\cP|}
\] 
to derive the required result.

Let $r$ be a line of $\cQ(4, q)$. If $r \in \cL_{q+1}$, then it has one point in common with both $\cX$ and $\cX'$, $q-1$ points of $\cY \setminus (\cX \cup \cX')$ and no point of $\cZ_i$. If $r \in \cL_2$, it has one point in common with both $\cX$ and $\cX'$, $(q-1)/(2+\xi)$ points of $\cZ_i$ and no point of $\cY \setminus (\cX \cup \cX')$. 

Assume $q \equiv -1 \pmod 3$. If $r \in \cL_0$, then all its points belong to $\cZ_1$, while if $r \in \cL_1$, then $|r \cap (\cY \setminus (\cX \cup \cX'))| = 1$ and $|r \cap \cZ_1| = q$. If $r \in \cL_3$, then $|r \cap (\cY \setminus (\cX \cup \cX')) = 3$ and $|r \cap \cZ_1| = q-2$.

In the case when $q \equiv 1 \pmod 3$, a more refined argument is needed.

\begin{lemma}\label{technical}
Let $q \equiv 1 \pmod{3}$. A point $P$ belongs to $\cZ_i$ if and only if $F_1(P) = \w^{3j+i-1}$. 
\end{lemma}
\begin{proof}
First of all observe that $F_1(\lambda P) = \lambda^3 F_1(P)$. Since $\cZ_i = R_i^{\tilde{G}_1}$, where $R_{i} = U_1 + \w^{i-1} U_5$, $i = 1,2,3$, it is enough to see that if $g \in \tilde{G}_1$ is represented by $\tilde{M}_{a,b,c,d}$, then $F_1(R_i) = F_1(R_i^g) = \w^{2(i-1)}$.
\end{proof}

\begin{prop}
Let $q = 2^n \equiv 1 \pmod{3}$. A line of $\cL_3$ has $3$ points in common with $\cY$, $(q - (-1)^{\frac{n}{2}} 2 \sqrt{q} -2)/3$ points in common with $\cZ_1$ and $(q + (-1)^{\frac{n}{2}} \sqrt{q} - 2)/3$ in common with both $\cZ_2$ and $\cZ_3$. A line of $\cL_0$ has $(q - (-1)^{\frac{n}{2}} 2 \sqrt{q} + 1)/3$ points in common with $\cZ_1$ and $(q + (-1)^{\frac{n}{2}} \sqrt{q} + 1)/3$ in common with both $\cZ_2$ and $\cZ_3$. 
\end{prop}
\begin{proof}
Recall that the points of a line of $\cQ(4, q)$ in $\cL_0$ or $\cL_3$ are the image under the Pl\"ucker embedding of the lines of $\cW(3, q)$ through a point $Q$ belonging to $\cO_0$ or $\cO_3$, respectively. Since $q \equiv 1 \pmod{3}$, by Proposition~\ref{points}, we may assume $Q = (1,0,0,x)$, where $x \in \F_q$ is not a cube, or $Q = (1,0,0,1)$, respectively. Therefore a representative for $\cL_0$ is the line 
\begin{align*}
    & r_x = \{(1,\mu,0,0,x,\mu x) \mid \mu \in \F_q\} \cup \{U_2 + x U_6\},
\end{align*}
whereas for $\cL_3$ is the line 
\begin{align*}
    & r = \{(1,\mu,0,0,1,\mu) \mid \mu \in \F_q\} \cup \{U_2 + U_6\}.    
\end{align*}
Taking into account Lemma~\ref{technical}, $|r_x \cap \cZ_i|$ equals the number of points of $r_x$ whose evaluation of $F$ has the form $\w^{3j+i-1}$, for some $j$. In particular, the point $U_2 + x U_6$ belongs to $\cZ_2$ or $\cZ_3$, according as $x$ equals $\w^{3l+1}$ or $\w^{3l+2}$. The point $P_\mu = (1,\mu,0,0,x,\mu x)$ belongs to $\cZ_1$ if and only if there exists $\lambda \in \F_q$ such that
\begin{align}
    & x \mu^3 + \lambda^3 = x^2. \label{eq5}
\end{align}
By \cite[Theorem 1]{W}, in both cases $x = \w^{3l+1}$ or $x = \w^{3l+2}$, the number of pairs $(\mu, \lambda) \in \F_q^2$ satisfying \eqref{eq5} is $q - (-1)^{\frac{n}{2}}2\sqrt{q} + 1$. Since for each $\mu$, there are three distinct $\lambda$ satisfying \eqref{eq5}, it follows that 
\begin{align*}
    & |r_x \cap \cZ_1| = \frac{q - (-1)^{\frac{n}{2}}2\sqrt{q} + 1}{3}.
\end{align*}
Observe that $P_\mu \in r_x \cap \cZ_2$ if and only if $P_{\mu^2} \in r_{x^2} \cap \cZ_3$. Hence $|r_x \cap \cZ_2| = |r_{x^2} \cap \cZ_3|$. On the other hand $|r_x \cap \cZ_3| = |r_{x^2} \cap \cZ_3|$, since $r_x$ and $r_{x^2}$ are in the same $\tilde{G}_h$-orbit. Therefore
\begin{align*}
    & |r_x \cap \cZ_2| = |r_x \cap \cZ_3| = \frac{q+1 - |r_0 \cap \cZ_1|}{2} = \frac{q + (-1)^{\frac{n}{2}}\sqrt{q} + 1}{3}.
\end{align*}

Similarly, the point $U_2 + U_6$ belongs to $\cZ_1$, whereas the point $P_\mu = (1,\mu,0,0,1,\mu)$ belongs to $\cZ_1$ if and only if $\mu^3 \ne 1$ and there exists $\lambda \in \F_q$ such that
\begin{align}
    & \mu^3 + \lambda^3 = 1. \label{eq6}
\end{align}
By using again \cite[Theorem 1]{W} (or \cite[Corollary 4]{W}), we find that the number of pairs $(\mu, \lambda) \in \F_q^2$ satisfying \eqref{eq6} is $q - (-1)^{\frac{n}{2}}2\sqrt{q} - 2$. Among these solutions we have to exclude $(1, 0)$, $(\w^{\frac{q-1}{3}}, 0)$, $(\w^{\frac{2(q-1)}{3}}, 0)$ and for each $\mu$ of the remaining pairs, there are three distinct $\lambda$ satisfying \eqref{eq6}, it follows that 
\begin{align*}
    & |r \cap \cZ_1| = \frac{q - (-1)^{\frac{n}{2}}2\sqrt{q} - 2}{3}.
\end{align*}
Finally $P_\mu \in r \cap \cZ_2$ if and only if $P_{\mu^2} \in r \cap \cZ_3$. Hence 
\begin{align*}
    & |r \cap \cZ_2| = |r \cap \cZ_3| = \frac{q+1 - 3 - |r_0 \cap \cZ_1|}{2} = \frac{q + (-1)^{\frac{n}{2}}\sqrt{q} - 2}{3}.
\end{align*}
\end{proof}

Therefore the number of lines of $\cL_i$ through a point of $\cQ(4, q)$ of a given $\tilde{G}_1$-orbit on points of $\cQ(4, q)$ can be deduced as shown in Tables~\ref{-1} and \ref{+1}.

\begin{table}[ht]\caption{Number of lines of $\cL_i$ through a point of $\cQ(4, q)$, $q \equiv -1 \pmod{3}$}\label{-1} 
\centering 
\begin{tabular}{|c|c|c|c|c|}
\hline 
$\quad$ & $\cX$ & $\cX'$ & $\cY \setminus (\cX \cup \cX')$ & $\cZ_1$ \\
\hline 
$\cL_{q+1}$ & $1$ & $1$ & $1$ & $0$ \\
$\cL_0$ & $0$ & $0$ & $0$ & $(q+1)/3$ \\
$\cL_1$ & $0$ & $0$ & $q/2$ & $q/2$ \\
$\cL_2$ & $q$ & $q$ & $0$ & $1$ \\
$\cL_3$ & $0$ & $0$ & $q/2$ & $(q-2)/6$ \\
\hline
\end{tabular} 
\end{table}

\begin{table}[ht]\caption{Number of lines of $\cL_i$ through a point of $\cQ(4, q)$, $q \equiv 1 \pmod{3}$}\label{+1} 
\centering 
\begin{tabular}{|c|c|c|c|c|c|c|}
\hline 
$\quad$ & $\cX$ & $\cX'$ & $\cY \setminus (\cX \cup \cX')$ & $\cZ_1$ & $\cZ_2$ & $\cZ_3$ \\ 
\hline 
$\cL_{q+1}$ & $1$ & $1$ & $1$ & $0$ & $0$ & $0$ \\
$\cL_0$ & $0$ & $0$ & $0$ & $(q - (-1)^{\frac{n}{2}} 2 \sqrt{q} + 1)/3$ & $(q + (-1)^{\frac{n}{2}} \sqrt{q} + 1)/3$ & $(q + (-1)^{\frac{n}{2}} \sqrt{q} + 1)/3$ \\
$\cL_1$ & $0$ & $0$ & $q/2$ & $(q + (-1)^{\frac{n}{2}}2\sqrt{q})/2$ & $(q - (-1)^{\frac{n}{2}}\sqrt{q})/2$ & $(q - (-1)^{\frac{n}{2}}\sqrt{q})/2$ \\
$\cL_2$ & $q$ & $q$ & $0$ & $1$ & $1$ & $1$ \\
$\cL_3$ & $0$ & $0$ & $q/2$ & $(q - (-1)^{\frac{n}{2}} 2 \sqrt{q} -2)/6$ & $(q + (-1)^{\frac{n}{2}} \sqrt{q} - 2)/6$ & $(q + (-1)^{\frac{n}{2}} \sqrt{q} - 2)/6$ \\
\hline
\end{tabular} 
\end{table}

\subsubsection{The number of lines of $\cL_i$ in a regulus of $\cQ(4, q)$}

By Theorem~\ref{klein}, the group $\tilde{G}_h$ has $2(q+1)$ orbits on reguli of $\cQ(4, q)$. Denote by $\cS$ one of these $2(q+1)$ orbits and by $\Sigma$ a solid of $\Pi$ such that 
\begin{itemize}
\item[{\em i)}] $\Sigma \cap \cQ(4, q)$ is a three-dimensional hyperbolic quadric $\cQ^+$;
\item[{\em ii)}] exactly one of the two reguli $\cR_1$, $\cR_2$ of $\cQ^+$ belongs to $\cS$. 
\end{itemize}
Recall that $\cY$ consists of the $q+1$ lines 
\begin{align*}
    & \ell_x = \{(1, \lambda, x^2+\lambda x, x^2+\lambda x, \lambda x^2, x^4) \mid \lambda \in \F_q\} \cup \{(0,1,x,x,x^2,0)\}, \; x \in \F_q, \\
    & \ell_\infty = \{(0,0,0,0,1, \lambda) \mid \lambda \in \F_q\} \cup \{U_6\}.
\end{align*}

\begin{theorem}\label{AllReguli}
Let $\zeta$ be a fixed element in $\F_{q^2} \setminus \F_q$ such that $(\z+1)^{q+1} = 1$. With the above notation, one of the following cases occurs.

\begin{enumerate}

\item If $\cS$ is one of the two orbits of size $q(q+1)/2$, then
\begin{itemize}
	\item $|\cL_{q+1} \cap \cR_1| = 2$, $|\cL_2 \cap \cR_1| = 0$, 
		\begin{align*}
			& |\cL_0 \cap \cR_1| = 2(q-1)/3, |\cL_1 \cap \cR_1| = 0, |\cL_3 \cap \cR_1| = (q-1)/3, && \mbox{ if } q \equiv 1 \pmod{3}, \\
			& |\cL_0 \cap \cR_1| = 0, |\cL_1 \cap \cR_1| = q-1, |\cL_3 \cap \cR_1| = 0, && \mbox{ if } q \equiv -1 \pmod{3}.	
		\end{align*}
	\item $|\cL_{q+1} \cap \cR_2| = |\cL_0 \cap \cR_2| = |\cL_1 \cap \cR_2| = 0$, $|\cL_2 \cap \cR_2| = 2$, $|\cL_3 \cap \cR_2| = q-1$.
\end{itemize}

\item If $\cS$ is one of the two orbits of size $q^3-q$, then 
\begin{itemize}
	\item $|\cL_{q+1} \cap \cR_1| = 1$, $|\cL_2 \cap \cR_1| = 1$, 
		\begin{align*}
			& |\cL_0 \cap \cR_1| = (q-1)/3, |\cL_1 \cap \cR_1| = q/2, |\cL_3 \cap \cR_1| = (q-4)/6, && \mbox{ if } q \equiv 1 \pmod{3}, \\
			& |\cL_0 \cap \cR_1| = (q+1)/3, |\cL_1 \cap \cR_1| = (q-2)/2, |\cL_3 \cap \cR_1| = (q-2)/6, && \mbox{ if } q \equiv -1 \pmod{3}.	
		\end{align*}
	\item $|\cL_{q+1} \cap \cR_2| = |\cL_0 \cap \cR_2| = 0, |\cL_1 \cap \cR_2| = q/2$, $|\cL_2 \cap \cR_2| =2$, $|\cL_3 \cap \cR_2| = (q-2)/2$.
\end{itemize}

\item If $\cS$ is one of the $q-2$ orbits with $|\cS| = (q^3-q)/2$ and $|\Sigma \cap \cX| = 2$, let 
\begin{align*}
	& \u_i = |\{x \in \F_q \setminus \{0,z+1,(z+1)^{-1}\} \mid \Tr_{q|2}\left(f_i(x)\right) = 0\}|, 
\end{align*}
where 
\begin{align*}
	& f_1(x) = \frac{1}{z^2} \left(\frac{x(z+1)}{z^2}+\frac{z+1}{xz^2}+1\right), & f_2(x) = \frac{z^2+1}{z^2} \left(\frac{x(z+1)}{z^2}+\frac{z+1}{xz^2}+1\right), 
\end{align*}	
	for some $z \in \F_q \setminus \{0,1\}$. Then
\begin{align*}
	& |\cL_{q+1} \cap \cR_i| = 0, |\cL_0 \cap \cR_i| = (2\u_i+6)/3, |\cL_1 \cap \cR_i| = q-\u_i-3, |\cL_2 \cap \cR_i| = 2, |\cL_3 \cap \cR_i| = \u_i/3, \\
	& i \in \{1, 2\}. 
\end{align*}

\item If $\cS$ is one of the two orbits of size $q(q-1)/2$, then
\begin{itemize}
	\item $|\cL_{q+1} \cap \cR_1| = |\cL_2 \cap \cR_1| = 0$, 
		\begin{align*}
			& |\cL_0 \cap \cR_1| = |\cL_3 \cap \cR_1| = 0, |\cL_1 \cap \cR_1| = q+1, && \mbox{ if } q \equiv 1 \pmod{3}, \\
			& |\cL_0 \cap \cR_1| = 2(q+1)/3, |\cL_1 \cap \cR_1| = 0, |\cL_3 \cap \cR_1| = (q+1)/3, && \mbox{ if } q \equiv -1 \pmod{3}.	
		\end{align*}
	\item $|\cL_{q+1} \cap \cR_2| = |\cL_0 \cap \cR_2| = |\cL_2 \cap \cR_2| = |\cL_3 \cap \cR_2| = 0$, $|\cL_1 \cap \cR_2| = q+1$.
\end{itemize}

\item If $\cS$ is one of the $q-2$ orbits with $|\cS| = (q^3-q)/2$ and $|\Sigma \cap \cX| = 0$, let 
\begin{align*}
	& \v_i = |\{x \in \F_q \mid \Tr_{q|2}\left(g_i(x)\right) = 0\}|, 
\end{align*}
where 
\begin{align*}
	& g_1(x) = \frac{\frac{\z^{2(q+1)}}{t^{2(q+1)}} \left( x^2 + \frac{t^{q+1}}{\z^{q+1}}x + \frac{t^2}{t^{q+1}} + \frac{t^{q+1}}{\z^2} \right) \left( \left( \frac{t^2}{t^{q+1}} + \frac{t^{q+1}}{\z^2} \right)x^2 + \frac{t^{q+1}}{\z^{q+1}}x + 1\right)}{(x^2+\z^{q+1}+1)^2}, \\
	& g_2(x) = \frac{\frac{\z^{2(q+1)}}{t^{2(q+1)}} \left( x^2 + \frac{t^{q+1}}{\z^{q+1}}x + \frac{t^{q+1}}{t^{2}} + \frac{t^{q+1}}{\z^2} \right) \left( \left( \frac{t^{q+1}}{t^{2}} + \frac{t^{q+1}}{\z^2} \right)x^2 + \frac{t^{q+1}}{\z^{q+1}}x + 1 \right)}{(x^2+\z^{q+1}+1)^2},
\end{align*}	
	for some $t \in \F_{q^2} \setminus \left\{0, \z^2, \z^{2q}\right\}$, $(t+1)^{q+1} = 1$. Then
\begin{align*}
	& |\cL_{q+1} \cap \cR_i| = 0, |\cL_0 \cap \cR_i| = (2\v_i+6)/3, |\cL_1 \cap \cR_i| = q-\v_i-2, |\cL_2 \cap \cR_i| = 0, |\cL_3 \cap \cR_i| = (\v_i+3)/3, \\
	& \mbox{ or } \\
	& |\cL_{q+1} \cap \cR_i| = 0, |\cL_0 \cap \cR_i| = 2\v_i/3, |\cL_1 \cap \cR_i| = q-\v_i+1, |\cL_2 \cap \cR_i| = 0, |\cL_3 \cap \cR_i| = \v_i/3,  \\
	& i \in \{1, 2\}.
\end{align*}
according as $\Tr_{q|2} \left( \frac{z^2}{t^{q+1}} \left(1 + \frac{z^{2q}}{t^2}\right) \right)$ equals $0$ or $1$, respectively.
\end{enumerate}
\end{theorem}
\begin{proof}
For each of the five cases listed in Corollary~\ref{representatives}, we determine $\cQ^+ \cap \cY$ and hence $|\cL_i \cap \cR_j|$, $i \in \{0, 1, 2, 3, q+1\}$, $j \in \{1,2\}$.

If $\cS$ is one of the two orbits of size $q(q+1)/2$ then
\begin{align*}
    & \cQ^+ \cap \cY = \{(1,x,0,0,x^3,x^4) \mid x \in \F_q \setminus \{0\}\} \cup \ell_0 \cup \ell_\infty.
\end{align*}
Assume that $\ell_0 \in \cR_1$. Then $|\cL_{q+1} \cap \cR_1| = |\cL_2 \cap \cR_2| = 2$ and $|\cL_{q+1} \cap \cR_2| = |\cL_2 \cap \cR_1| = 0$. We infer that $|\cL_{3} \cap \cR_2| = q-1$ and $|\cL_1 \cap \cR_2| = |\cL_0 \cap \cR_2| = 0$. The two points $(1,x,0,0,x^3,x^4)$, $(1,y,0,0,y^3,y^4)$, $x, y \in \F_q \setminus \{0\}$, $x \ne y$, belonging to $\cY \setminus (\cX \cup \cX')$ define a line of $\cL_3 \cap \cR_1$ if and only if 
\begin{align*}
    & \rk 
    \begin{pmatrix}
        1 & x & 0 & 0 & x^3 & x^4 \\
        1 & y & 0 & 0 & y^3 & y^4 \\
        1 & 0 & 0 & 0 & 0 & 0 \\
        0 & 0 & 0 & 0 & 1 & 0 
    \end{pmatrix} = 3,
\end{align*}
which gives $xy(x^3+y^3) = 0$. For $q \equiv 1 \mod 3$, $|\cL_3 \cap \cR_1| = (q-1)/3$, $|\cL_1 \cap \cR_1| = 0$ and $|\cL_0 \cap \cR_1| = 2(q-1)/3$. For $q \equiv -1 \mod 3$, instead, it holds  $|\cL_3 \cap \cR_1| = |\cL_0 \cap \cR_1| = 0$ and $|\cL_1 \cap \cR_1| = q-1$. 

Let $\cS$ be one of the two orbits of size $q^3-q$, then 
\begin{align*}
    & \cQ^+ \cap \cY = \{(x,x^2+1,x,x,x^2+x^4,x^5) \mid x \in \F_q\} \cup \ell_\infty.
\end{align*}
Assume that $\ell_\infty \in \cR_1$ and $\langle U_2, U_6 \rangle \in \cR_2$. Then $|\cL_{q+1} \cap \cR_1| = 1$, $|\cL_2 \cap \cR_1| = 1$ and $|\cL_{q+1} \cap \cR_2| = 0$, $|\cL_2 \cap \cR_2| = 2$. Let $s$ be the line defined by the two points $(x,x^2+1,x,x,x^2+x^4,x^5)$, $(y,y^2+1,y,y,y^2+y^4,y^5)$, $x, y \in \F_q \setminus \{0, 1\}$, $x \ne y$, of $\cY \setminus (\cX \cup \cX')$. The line $s$ belongs to $\cL_3 \cap \cR_1$ if and only if 
\begin{align*}
    & \rk 
    \begin{pmatrix}
        x & x^2+1 & x & x & x^2+x^4 & x^5 \\
        y & y^2+1 & y & y & y^2+y^4 & y^5 \\
        0 & 1 & 0 & 0 & 0 & 0 \\
        0 & 0 & 0 & 0 & 0 & 1 
    \end{pmatrix} = 3,
\end{align*}
which gives $xy(x+y)(y^2+xy+x^2+1) = 0$. Since 
\begin{align*}
    & \left\vert\left\{x \in \F_q \setminus \{0, 1\} \mid \Tr\left(1 + \frac{1}{x^2}\right) = 0\right\}\right\vert = 
    \begin{cases}
        \frac{q}{2} - 2 & \mbox{ if } q \equiv 1 \pmod{3}, \\
        \frac{q}{2} - 1 & \mbox{ if } q \equiv -1 \pmod{3}, 
    \end{cases}
\end{align*}
it follows that $|\cL_3 \cap \cR_1| = (q-4)/6$, $|\cL_1 \cap \cR_1| = q/2$ and $|\cL_0 \cap \cR_1| = (q-1)/3$, if $q \equiv 1 \pmod{3}$, whereas $|\cL_3 \cap \cR_1| = (q-2)/6$, $|\cL_1 \cap \cR_1| = (q-2)/2$ and $|\cL_0 \cap \cR_1| = (q+1)/3$, if $q \equiv -1~\pmod{3}$. The line $s$ belongs to $\cL_3 \cap \cR_2$ if and only if 
\begin{align*}
    & \rk 
    \begin{pmatrix}
        x & x^2+1 & x & x & x^2+x^4 & x^5 \\
        y & y^2+1 & y & y & y^2+y^4 & y^5 \\
        0 & 0 & 0 & 0 & 1 & 0 \\
        0 & 0 & 0 & 0 & 0 & 1 
    \end{pmatrix} = 3,
\end{align*}
which gives $(xy+1)(x+y) = 0$. Hence $|\cL_3 \cap \cR_2| = (q-2)/2$, $|\cL_1 \cap \cR_2| = q/2$ and $|\cL_0 \cap \cR_2| = 0$.  

Assume that $\cS$ is one of the $q-2$ orbits with $|\cS| = (q^3-q)/2$ and $|\Sigma \cap \cX| = 2$, then 
\begin{align*}
    \cQ^+ \cap \cY = & \left\{T_x \mid x \in \F_q \setminus \{0,z+1,(z+1)^{-1}\}\right\} \\
    & \cup \left\{U_2, U_5, (1,0,z^2+1,z^2+1,0,z^4+1), (1,0,(z^2+1)^{-1}, (z^2+1)^{-1},0,(z^4+1)^{-1})\right\}, 
\end{align*}    
where
\begin{align*}    
    T_x = \frac{z^2+1}{z^4} & \left(\frac{z^4}{z^2+1}, x^{-1}\left(x^4+\frac{z^4}{z^2+1}x^2+1\right), x^4+1, x^4+1, x \left(x^4+\frac{z^4}{z^2+1}x^2+1\right), \frac{z^4}{z^2+1} x^4\right),
\end{align*} 
for some $z \in \F_q \setminus \{0,1\}$. Suppose that $\langle (1,0,z^2+1,z^2+1,0,z^4+1), U_2 \rangle \in \cR_1$ and $\langle (1,0,(z^2+1)^{-1},(z^2+1)^{-1},0,(z^4+1)^{-1}), U_2 \rangle \in \cR_2$. In this case $|\cL_{q+1} \cap \cR_1| = |\cL_{q+1} \cap \cR_2| = 0$ and $|\cL_2 \cap \cR_1| = |\cL_2 \cap \cR_2| = 2$. The line $\langle T_x, T_y \rangle$, $x, y \in \F_q \setminus \{0,z+1,(z+1)^{-1}\}$, $x \ne y$, belongs to $\cL_3 \cap \cR_2$ or to $\cL_3 \cap \cR_1$ if and only if 
\begin{align}
    & (x+z+1)^2 y^2 + \frac{x z^4}{z^2+1} y + (x(z+1)+1)^2 = 0 \mbox{ or } \label{eq9} \\
    & (x(z+1)+1)^2 y^2 + x z^4 y + (x+z+1)^2 = 0. \label{eq10}
\end{align}
For a fixed $x$, equation \eqref{eq9} or \eqref{eq10} has a solution in $y$ if and only if 
\begin{align*}
    & \Tr_{q|2}\left( \frac{z^2+1}{z^2} \left(\frac{x(z+1)}{z^2}+\frac{z+1}{xz^2}+1\right) \right) = 0 \mbox{ or } \Tr_{q|2}\left( \frac{1}{z^2} \left(\frac{x(z+1)}{z^2}+\frac{z+1}{xz^2}+1\right) \right) = 0.
\end{align*}

Let $\cS$ be one of the two orbits of size $q(q-1)/2$. Then
\begin{align*}
    & \cQ^+ \cap \cY = \{T_x \mid x \in \F_q\} \cup \{T_\infty = (0,0,0,0,1,\z^{q+1})\}, \\
    & \mbox{where } T_x = (\z^{q+1}, x^2+\z^{q+1}x+1, x^3+x, x^3+x, x^4+\z^{q+1}x^3+x^2, \z^{q+1}x^4).
\end{align*}
Let $\bar{\cQ}^+$ denote the quadratic extension of $\cQ^+$, and assume that $\langle (0,1,\z+1,\z+1,\z^2+1,0), (1,0,\z^2+1,\z^2+1,0,\z^4+1) \rangle \in \bar{\cR}_1$ and $\langle (0,1,\z+1,\z+1,\z^2+1,0), (1,0,\z^{2q}+1,\z^{2q}+1,0,\z^{4q}+1) \rangle \in \bar{\cR}_2$, where $\bar{\cR}_1$ and $\bar{\cR}_2$ are the reguli of $\bar{\cQ}^+$ obtained by extending $\cR_1$ and $\cR_2$. In this case $|\cL_{q+1} \cap \cR_1| = |\cL_2 \cap \cR_1| = |\cL_{q+1} \cap \cR_2| = |\cL_2 \cap \cR_2| = 0$. Arguing as before we find that the line joining $T_x$ and $T_\infty$ belongs to $\cL_3 \cap \cR_1$ if and only if there exists $x \in \F_q$ such that
\begin{align}
    & (\z+1)^2x^2 + \z^2(\z+1)x + \z^4+\z^2+1 = (\z+1)^2\left((x^2+\z^{q+1}x+(\z^{q+1}+1)^2\right) = 0. \label{eq7}   
\end{align}
Since $u^2+\z^{q+1}u+1 = u^2+(\z+\z^q)u+1 = (u+\z+1)(u+\z^q+1)$, then $\Tr_{q|2}\left(1+\frac{1}{\z^{2(q+1)}}\right) = \Tr_{q|2}(1) + \Tr_{q|2}(1/\z^{2(q+1)}) = \Tr_{q|2}(1) + 1$. Hence equation \eqref{eq7} has two or none solutions according as $q \equiv -1$ or $1 \pmod{3}$. Similarly, for a fixed $x \in \F_q$, the line joining $T_x$ and $T_y$ belongs to $\cL_3 \cap \cR_1$ if and only if there exists $y \in \F_q \setminus \{x\}$ such that
\begin{align}
    &\left((\z+1)^2x^2 + \z^2(\z+1)x + \z^4+\z^2+1\right)y^2 + \z^2 \left((\z+1)x^2+\z^2 x+ \z+1\right) y \nonumber \\
    & + \left((\z^4+\z^2+1) x^2 + \z^2(\z+1) x + (\z+1)^2 \right) = & \nonumber \\
    & (\z+1)^2 \left( \left(x^2+\z^{q+1}x+(\z^{q+1}+1)^2\right) y^2 + \z^{q+1}\left(x^2+\z^{q+1}x+1\right) y + (\z^{q+1}+1)^2 x^2 + \z^{q+1}x + 1 \right) \nonumber \\
    & = 0. \label{eq8}
\end{align}
Since $u^2+(x^2+\z^{q+1}x+1)u+\z^{q+1}(x^3+x) = u^2+(x^2+(\z+\z^q)x+1)u+(\z+\z^q)(x^3+x) = (u+(\z+\z^q)x)(u+x^2+1)$, then $\Tr_{q|2}\left(1+\frac{1}{\z^{2(q+1)}}+\frac{\z^{q+1}(x^3+x)}{(x^2+\z^{q+1}x+1)^2}\right) = \Tr_{q|2}(1) + \Tr_{q|2}(1/\z^{2(q+1)}) + \Tr_{q|2}(\z^{q+1}(x^3+x)/(x^2+\z^{q+1}x+1)^2) = \Tr_{q|2}(1) + 1 + 0$. Therefore equation \eqref{eq8} has two or none solutions according as $q \equiv -1$ or $1 \pmod{3}$. It follows that $|\cL_3 \cap \cR_1| = (q+1)/3$, $|\cL_1 \cap \cR_1| = 0$ and $|\cL_0 \cap \cR_1| = 2(q+1)/3$, if $q \equiv -1 \pmod{3}$, whereas $|\cL_3 \cap \cR_1| = |\cL_0 \cap \cR_1| = 0$ and $|\cL_1 \cap \cR_1| = q+1$, if $q \equiv 1 \pmod{3}$. Similar calculations show that no line joining $T_x$ and $T_\infty$ or $T_x$ and $T_y$ belongs to $\cL_3 \cap \cR_2$. Hence $|\cL_3 \cap \cR_2| = 0$, $|\cL_1 \cap \cR_2| = q+1$ and $|\cL_0 \cap \cR_2| = 0$.

If $\cS$ is one of the $q-2$ orbits with $|\cS| = (q^3-q)/2$ and $|\Sigma \cap \cX| = 0$, then 
\begin{align*}
    \cQ^+ \cap \cY = & \left\{T_x \mid x \in \F_q\right\} \cup \left\{T_\infty = \left(0,0,0,0,1, t^{q+1}/\z^{q+1}\right)\right\}, 
\end{align*}
where 
\begin{align*}
T_x = & \left( ( t^{q+1}(x^2+\z^{q+1}x+1), \z^{q+1}(x^4+t^{q+1}x^2+1), (x^3+x)(\z^{q+1}x^2+t^{q+1}x+\z^{q+1}), \right. \\
    & \left. (x^3+x)(\z^{q+1}x^2+t^{q+1}x+\z^{q+1}), \z^{q+1}x^2(x^4+t^{q+1}x^2+1), t^{q+1}x^4(x^2+\z^{q+1}x+1) \right),
\end{align*}
for some $t \in \F_{q^2} \setminus \left\{0, \z^2, \z^{2q}\right\}$, $(t+1)^{q+1} = 1$. Let $\bar{\cQ}^+$ denote the quadratic extension of $\cQ^+$, and assume that $\langle (0,1,\z+1,\z+1,\z^2+1,0), (1,0,t+1,t+1,0,t^2+1) \rangle \in \bar{\cR}_1$ and $\langle (0,1,\z+1,\z+1,\z^2+1,0), (1,0,t^{q}+1,t^{q}+1,0,t^{2q}+1) \rangle \in \bar{\cR}_2$, where $\bar{\cR}_1$ and $\bar{\cR}_2$ are the reguli of $\bar{\cQ}^+$ obtained by extending $\cR_1$ and $\cR_2$. In this case $|\cL_{q+1} \cap \cR_1| = |\cL_2 \cap \cR_1| = |\cL_{q+1} \cap \cR_2| = |\cL_2 \cap \cR_2| = 0$. The line joining $T_x$ and $T_\infty$ belongs to $\cL_3 \cap \cR_1$ if and only if there exists $x \in \F_q$ such that
\begin{align}
    \z^2(t+1)x^2+t^2(\z+1)x+\z^2t^2+\z^2+t^2 & = \z^2(t+1) \left(x^2 + \frac{t^{q+1}}{\z^{q+1}}x + \frac{t^2}{t^{q+1}} + \frac{t^{q+1}}{\z^2}\right) \nonumber \\
    & = \z^2(t+1) \left(x^2 + \frac{t^{q+1}}{\z^{q+1}}x + \frac{t^{2q}}{t^{q+1}} + \frac{t^{q+1}}{\z^{2q}}\right) = 0 \label{eq11}   
\end{align}
Equation \eqref{eq11} has two or none solutions according as $\Tr_{q|2}\left(\frac{\z^2}{t^{q+1}} \left(1 + \frac{\z^{2q}}{t^2}\right)\right)$ equals $0$ or $1$, respectively. Analogously, for a fixed $x \in \F_q$, the line joining $T_x$ and $T_y$ belongs to $\cL_3 \cap \cR_1$ if and only if there exists $y \in \F_q \setminus \{x\}$ such that
\begin{align}
    & \left(x^2 + \frac{t^{q+1}}{\z^{q+1}}x + \frac{t^2}{t^{q+1}} + \frac{t^{q+1}}{\z^2}\right) y^2 + \frac{t^{q+1}}{\z^{q+1}} \left(x^2+\z^{q+1}x+1\right) y + \left( \left( \frac{t^2}{t^{q+1}} + \frac{t^{q+1}}{\z^2} \right)x^2 + \frac{t^{q+1}}{\z^{q+1}}x + 1\right) \nonumber \\
    & = 0. \label{eq12}
\end{align}
Equation \eqref{eq12} admits a solution if and only if
\begin{align*}
    & \Tr_{q|2} \left( \frac{\frac{\z^{2(q+1)}}{t^{2(q+1)}} \left( x^2 + \frac{t^{q+1}}{\z^{q+1}}x + \frac{t^2}{t^{q+1}} + \frac{t^{q+1}}{\z^2} \right) \left( \left( \frac{t^2}{t^{q+1}} + \frac{t^{q+1}}{\z^2} \right)x^2 + \frac{t^{q+1}}{\z^{q+1}}x + 1\right)}{(x^2+\z^{q+1}+1)^2} \right) = 0.
\end{align*}
The line joining $T_x$ and $T_\infty$ belongs to $\cL_3 \cap \cR_2$ if and only if there exists $x \in \F_q$ such that
\begin{align*}
    & x^2 + \frac{t^{q+1}}{\z^{q+1}}x + \frac{t^{q+1}}{\z^2}+\frac{t^{q+1}}{t^2} = 0    
\end{align*}
which has solutions if and only if $\Tr_{q|2}\left(\frac{\z^{2(q+1)}}{t^{q+1}} \left(\frac{1}{\z^2} + \frac{1}{t^2}\right)\right)$ equals $0$, whereas, for a fixed $x \in \F_q$, the line joining $T_x$ and $T_y$ belongs to $\cL_3 \cap \cR_2$ if and only if there exists $y \in \F_q \setminus \{x\}$ such that
\begin{align*}
    & \left(x^2 + \frac{t^{q+1}}{\z^{q+1}}x + \frac{t^{q+1}}{t^2} + \frac{t^{q+1}}{\z^2}\right) y^2 + \frac{t^{q+1}}{\z^{q+1}} \left(x^2+\z^{q+1}x+1\right) y + \left( \left( \frac{t^{q+1}}{t^{2}} + \frac{t^{q+1}}{\z^2} \right)x^2 + \frac{t^{q+1}}{\z^{q+1}}x + 1 \right) \nonumber \\
    & = 0. 
\end{align*}
which has solutions if and only if
\begin{align*}
    & \Tr_{q|2} \left( \frac{\frac{\z^{2(q+1)}}{t^{2(q+1)}} \left( x^2 + \frac{t^{q+1}}{\z^{q+1}}x + \frac{t^{q+1}}{t^{2}} + \frac{t^{q+1}}{\z^2} \right) \left( \left( \frac{t^{q+1}}{t^{2}} + \frac{t^{q+1}}{\z^2} \right)x^2 + \frac{t^{q+1}}{\z^{q+1}}x + 1 \right)}{(x^2+\z^{q+1}+1)^2} \right) = 0.
\end{align*}
Finally observe that 
\begin{align*}
\Tr_{q|2}\left(\frac{\z^2}{t^{q+1}} \left(1 + \frac{\z^{2q}}{t^2}\right)\right) = 0 \iff \Tr_{q|2}\left(\frac{\z^{2(q+1)}}{t^{q+1}} \left(\frac{1}{\z^2} + \frac{1}{t^2}\right)\right) = 0.
\end{align*}
\end{proof}

The content of Theorem~\ref{AllReguli} is summarized in Tables \ref{q(q+1)/2-}, \ref{q(q+1)/2+}, \ref{q3-q+1}, \ref{q3-q-1}, \ref{q2-q2-1}, \ref{q2-q2+1}, \ref{q3q22}  and \ref{q3q20R1}.   

\bigskip 

\begin{minipage}{8cm}
\centering 
\begin{tabular}{|c|c|c|}
\hline 
$\quad$ & $\cR_1$ & $\cR_2$\\
\hline 
$\cL_{q+1}$ & $2$ & $0$ \\
$\cL_0$ & $0$ & $0$  \\
$\cL_1$ & $q-1$ & $0$  \\
$\cL_2$ & $0$ & $2$  \\
$\cL_3$ & $0$ & $q-1$ \\
\hline
\end{tabular} 
\captionof{table}{ $\vert \cS \vert =q(q+1)/2$, $q \equiv -1 \mod 3$}\label{q(q+1)/2-} \end{minipage}
\hspace{1cm}
\begin{minipage}{8cm}
\centering 
\begin{tabular}{|c|c|c|}
\hline 
$\quad$ & $\cR_1$ & $\cR_2$\\
\hline 
$\cL_{q+1}$ & $2$ & $0$ \\
$\cL_0$ & $2(q-1)/3$ & $0$  \\
$\cL_1$ & $0$ & $0$  \\
$\cL_2$ & $0$ & $2$  \\
$\cL_3$ & $(q-1)/3$ & $q-1$ \\
\hline
\end{tabular} 
\captionof{table}{ $\vert \cS \vert =q(q+1)/2$, $q \equiv 1 \mod 3$}\label{q(q+1)/2+} \end{minipage}

 \begin{minipage}{8cm}
\centering 
\begin{tabular}{|c|c|c|}
\hline 
$\quad$ & $\cR_1$ & $\cR_2$\\
\hline 
$\cL_{q+1}$ & $1$ & $0$ \\
$\cL_0$ & $(q+1)/3$ & $0$  \\
$\cL_1$ & $(q-2)/2$ & $q/2$  \\
$\cL_2$ & $1$ & $2$  \\
$\cL_3$ & $(q-2)/6$ & $(q-2)/2$ \\
\hline
\end{tabular} \captionof{table}{$\vert \cS \vert =q^3-q$, in the case $q\equiv -1 \mod 3$}\label{q3-q-1} 
\end{minipage}
  \hspace{1cm}
 \begin{minipage}{8cm}
\centering 
\begin{tabular}{|c|c|c|}
\hline 
$\quad$ & $\cR_1$ & $\cR_2$\\
\hline 
$\cL_{q+1}$ & $1$ & $0$ \\
$\cL_0$ & $(q-1)/3$ & $0$  \\
$\cL_1$ & $q/2$ & $q/2$  \\
$\cL_2$ & $1$ & $2$  \\
$\cL_3$ & $(q-4)/6$ & $(q-2)/2$ \\
\hline
\end{tabular} 
\captionof{table}{ $\vert \cS \vert =q^3-q$, in the case $q\equiv 1 \mod 3$}\label{q3-q+1}
\end{minipage}

\begin{center}
\begin{tabular}{|c|c|}
\hline 
$\quad$ & $\cR_i$, $i \in \{1, 2\}$ \\
\hline 
$\cL_{q+1}$ & $0$ \\
$\cL_0$ & $(2\u_i+6)/3$   \\
$\cL_1$ & $q-\u_i-3 $  \\
$\cL_2$ & $2 $   \\
$\cL_3$ & $\u_i/3$  \\
\hline
\end{tabular} \captionof{table}{ $\vert \cS \vert =(q^3-q)/2$ in the case $\vert \Sigma \cap \cX\vert=2$}\label{q3q22}
\end{center}
 
 \begin{minipage}{8cm}
\centering 
\begin{tabular}{|c|c|c|}
\hline 
$\quad$ & $\cR_1$ & $\cR_2$\\
\hline 
$\cL_{q+1}$ & $0$ & $0$ \\
$\cL_0$ & $2(q+1)/3$ & $0$  \\
$\cL_1$ & $0$ & $q+1$  \\
$\cL_2$ & $0$ & $0$  \\
$\cL_3$ & $(q+1)/3$ & $0$ \\
\hline
\end{tabular} \captionof{table}{  $\vert \cS \vert =q(q-1)/2$ in the case $q \equiv -1 \mod 3$}\label{q2-q2-1}
\end{minipage}
 \hspace{1cm}
 \begin{minipage}{8cm}
\centering 
\begin{tabular}{|c|c|c|}
\hline 
$\quad$ & $\cR_1$ & $\cR_2$\\
\hline 
$\cL_{q+1}$ & $0$ & $0$ \\
$\cL_0$ & $0$ & $0$  \\
$\cL_1$ & $q+1$ & $q+1$  \\
$\cL_2$ & $0$ & $0$  \\
$\cL_3$ & $0$ & $0$ \\
\hline
\end{tabular} \captionof{table}{ $\vert \cS \vert = q(q-1)/2$ in the case $q \equiv 1 \mod 3$}\label{q2-q2+1} 
\end{minipage}

\begin{center} 
\centering 
\begin{tabular}{|c|c|c|}
\hline 
$\quad$ & $\Tr_{q|2} \left( \frac{z^2}{t^{q+1}} \left(1 + \frac{z^{2q}}{t^2}\right) \right)=0 $ & $\Tr_{q|2} \left( \frac{z^2}{t^{q+1}} \left(1 + \frac{z^{2q}}{t^2}\right) \right)=1$\\
$\quad$ & $\cR_i$, $i \in \{1, 2\}$ & $\cR_i$, $i \in \{1, 2\}$ \\
\hline 
$\cL_{q+1}$ & $0$ & $0$ \\
$\cL_0$ & $2(\v_i+3)/3$ & $2\v_i/3$  \\
$\cL_1$ & $ q-\v_i-2$ & $ q-\v_i+1$  \\
$\cL_2$ & $0$ & $0$  \\
$\cL_3$ & $(\v_i+3)/3$ & $ \v_i/3$ \\
\hline
\end{tabular} \captionof{table}{ $\vert \cS \vert =(q^3-q)/2$ in the case $\vert \Sigma \cap \cX\vert=0$}\label{q3q20R1} 
\end{center}

\bigskip

{\footnotesize
\noindent\textit{Acknowledgments.}
This work was supported by the Italian National Group for Algebraic and Geometric Structures and their Applications (GNSAGA-- INdAM).
}

\end{document}